\documentclass[11pt]{amsart}

\usepackage{amsmath,amssymb,amsthm}
\usepackage{amsmath,amssymb,amsthm,amscd}
\usepackage[frame,cmtip,arrow,matrix,line,graph,curve]{xy}
\usepackage{graphpap, color}
\usepackage[mathscr]{eucal}
\usepackage{color}

\numberwithin{equation}{section}

\usepackage{mathrsfs}%,mathabx}
\usepackage{graphicx}
\usepackage{pstricks}

\title[The K\"ahler cone and the balanced cone]
{Relations between the K\"ahler cone and the balanced cone of  a K\"ahler manifold}
\author{Jixiang Fu}
\author{Jian Xiao}
\address{Institute of Mathematics\\ Fudan University \\ Shanghai
200433, China}
\email{majxfu@fudan.edu.cn}
\email{jxiao10@fudan.edu.cn}
\newtheorem{prop}{Proposition}[section]
\newtheorem{theo}[prop]{Theorem}
\newtheorem{lemm}[prop]{Lemma}
\newtheorem{coro}[prop]{Corollary}
\newtheorem{rema}[prop]{Remark}

\newtheorem{defi}[prop]{Definition}
\newtheorem{conj}[prop]{Conjecture}

\begin{document}
\begin{abstract}
In this paper, we consider a natural map from the K\"ahler cone of a
compact K\"ahler manifold  to its balanced cone. We study its
injectivity and surjectivity. We also give an analytic
characterization theorem on a nef class being K\"ahler.
\end{abstract}
\maketitle

\section{Introduction}
On a complex $n$-dimensional manifold, a {\sl balanced metric} is a Hermitian metric such that its
associated fundamental form $\omega$ satisfies $d(\omega^{n-1})=0$.
Throughout this paper such an $\omega$ is called a balanced metric directly. It is easy to see that the existence of a balanced metric $\omega$ is equivalent to the
existence of a $d$-closed strictly positive $(n-1,n-1)$-form
$\Omega$ with the  relation  $\Omega=\omega^{n-1}$ (see \cite{M82}). Hence,
for convenience, each such $\Omega$ will also be called a balanced
metric.

Assume that $X$ is a compact complex manifold. The (real)
$(p,p)$-th Bott-Chern cohomology group of $X$ is defined as
$$H^{p,p}_{BC}(X,\mathbb R)=\{\textup{real $d$-closed
$(p,p)$-forms}\}/i\partial\bar\partial\{\textup {real
$(p-1,p-1)$-forms}\},$$
which can also be defined by currents. Its elements will be denoted by
$[\cdot]_{bc}$. It is easy to see that the cohomology classes of all real $(n-1,n-1)$-forms which are balanced metrics form an open convex cone in
$H^{n-1,n-1}_{BC}(X,\mathbb{R})$. We denote it by
\begin{eqnarray*}
  \mathcal{B}&=&  \{[\Omega]_{bc}\in H^{n-1,n-1}_{BC}(X,\mathbb{R})\ |\ \Omega \ \text{is a balanced
  metric}\}.
\end{eqnarray*}
It is called the {\sl balanced cone} of $X$. Note that the zero cohomology
class may be in  $\mathcal B$. For example, Fu-Li-Yau
\cite{FLY08} constructed a balanced metric $\omega$ on the connected
sum $Y$ of $k(\geq 2)$ copies of $S^3\times S^3$. Since
$H^{2,2}_{BC}\bigl( Y,\mathbb R\bigr)=0$,
$[\omega^{2}]_{bc}=0\in\mathcal B$. Clearly, if the zero class belongs to $\mathcal B$, then
$\mathcal B=H^{n-1,n-1}_{BC}(X,\mathbb R)$. However, if $X$ is
a compact K\"ahler manifold, then the zero class is never in $\mathcal B$.

Now we assume that $X$ is a compact K\"ahler manifold. In this case,
by the $\partial\bar\partial$-lemma, it is well known that
$H^{p,p}_{BC}(X,\mathbb R)$ is the same as  the cohomology group
$H^{p,p}_{dR}(X,\mathbb R)$, the set of de Rham classes
represented by a real $d$-closed $(p,p)$-form, see \cite{V02}. The
K\"ahler cone $\mathcal{K}$ of $X$ is defined to be
\begin{eqnarray*}
\mathcal{K} &=& { \{[\omega]\in
H^{1,1}_{dR}(X,\mathbb{R})\ |\ \omega\ \text{is a K\"ahler metric}\}},
\end{eqnarray*}
which  is  an open convex cone in $H^{1,1}_{dR}(X,\mathbb R)$. It  was studied thoroughly by Demailly and Paun in
\cite{DP04}. Since on a K\"ahler surface, the balanced cone and the
K\"ahler cone coincide by their definitions, we will always assume
 $n\geq 3$ in the rest of the paper.

The balanced cone $\mathcal{B}$ of a compact K\"ahler manifold is related to its movable cone $\mathcal{M}$ (cf. \cite{BDPP13} for its definition).
The first named author learned this notion from Professor Demailly, who mentioned
Toma's paper \cite{MT09}. Toma observed that every movable curve on a projective manifold can be represented by a balanced metric under the assumption $\mathcal{E}^ \vee =\overline{\mathcal{M}}$. This assumption is Conjecture 2.3 in \cite{BDPP13}. In fact, the result in \cite{MT09} holds for all movable classes on any compact K\"ahler manifold. And along the lines of \cite{MT09}, one can obtain
the equivalence of $\mathcal{B}$ and $\mathcal{M}$ under the assumption $\mathcal{E}^ \vee =\overline{\mathcal{M}}$ (see the appendix).
\vspace{2mm}

In this note, motivated by papers \cite{FWW10,FWW11}, we consider the map
$$\mathbf b:\mathcal K\to\mathcal B $$
which maps $[\omega]$ to $[\omega^{n-1}]$.
It is clearly well-defined and
can be extended to the map
$$\overline {\mathbf b}:\overline{\mathcal K}\to\overline{\mathcal B},$$
where $\overline{\mathcal K}$ and $\overline{\mathcal{B}}$ are the closures of the corresponding cones.
We want to study
the properties of $\mathbf b$ and $\overline{\mathbf b}$.  We will first prove  that
$\mathbf b$  embeds $\mathcal{K}$ into $\mathcal{B}$.
\begin{prop}
\label{injective map} Let $X$ be a compact K\"ahler manifold. Then
the map $\mathbf b$ is injective.
\end{prop}
The proof of the above proposition contains two key ingredients. The first one is Yau's celebrated theorem on
complex Monge-Amp\`ere equations over compact K\"ahler manifolds, and the second one is the
Arithmetric Mean-Geometric Mean (AM-GM) inequality. Replacing Yau's theorem by
Boucksom-Eyssidieux-Guedj-Zeriahi's theorem \cite{BEGZ} on
complex Monge-Amp\`ere equations in big cohomology classes, we can generalize the above proposition on the map $\mathbf b$ to the map $\overline{\mathbf b}$. Here we recall that in
the K\"ahler case a cohomology class $[\alpha]\in
H^{1,1}_{dR}(X,\mathbb R)$ is nef if $[\alpha]\in \overline{\mathcal
K}$, and $[\alpha]$ is big if $[\alpha]$ contains a K\"ahler current. For compact K\"ahler manifolds, Demailly and Paun \cite{DP04} proved that a nef class $[\alpha]$ is big if and only if $\int_X\alpha^n>0$. In order to generalize the above proposition, we also need a convexity inequality
obtained by  Gromov \cite{Gro90} and Demailly \cite{Dem93}, and use some techniques on currents.
\begin{theo}
\label{inj_thm}
Let $X$ be an $n$-dimensional compact K\"ahler manifold. Then the map
$\overline {\mathbf b}$ is
injective when $\overline {\mathbf b}$ is restricted to the subcone generated by all the nef and big classes.
\end{theo}
We remark that the condition
``big" is necessary, otherwise, the complex torus $T^n$ gives a counterexample.  But it is not clear whether the condition ``nef" is necessary.
\vspace{2mm}

In general $\mathbf b$ is not surjective. In fact, we will show that
$\overline{\mathbf b}(\partial\mathcal{K})\cap \mathcal{B}$ need not
to be empty. Let
$$\mathcal{K}_{NS}=\mathcal{K}\cap
NS_{\mathbb{R}},$$
where $NS_{\mathbb R}$ is the real Neron-Severi
group of $X$, i.e.,
$$NS_{\mathbb{R}}=\bigl(H^{1,1}_{BC}(X,\mathbb R)\cap H^2(X,\mathbb Z)_{\textup{free}}\bigr)\otimes _{\mathbb Z}\mathbb R.$$
Then, if $X$ is a projective Calabi-Yau manifold (i.e. a projective
manifold with $c_1=0$), we can characterize when a nef class
$[\alpha]\in \partial \mathcal K_{NS}$ can be mapped into
$\mathcal B$ by $\overline{\mathbf b}$.  In fact, inspired by the method in
\cite{MT09} and \cite{Sul76}, we can give some sufficient conditions when a $d$-closed
nonnegative $(n-1,n-1)$-form is a balanced class.
Applying these criteria to Proposition 4.1 in \cite{T07},
we obtain
\begin{theo}\label{nefb}
\label{PCY} Let $X$ be a projective Calabi-Yau manifold. If $[\alpha]
\in\partial\mathcal{K}$, then $\overline{\mathbf b}([\alpha])\in
\mathcal{B}$ implies that $[\alpha]$ is a big class. On the other hand, if
$[\alpha]\in \partial\mathcal{K}_{NS}$ is a big class,
then $\overline{\mathbf b}([\alpha])\in \mathcal{B}$ if and
only if the exceptional set $Exc(F_{[\alpha]})$ of the contraction
map $F_{[\alpha]}$ induced by $[\alpha]$ is of codimension greater than one, i.e., $F_{[\alpha]}$ is a flipping contraction.
\end{theo}
%If $[\alpha]\in NS(X)_{\mathbb{Q}}$, then the map $F_{[\alpha]}$ is the usual Kodaira map.

For general $[\alpha]\in NS(X)_{\mathbb{R}}$, the map $F_{[\alpha]}$ is described in \cite{T07}. In Section 3, we will give some details on it. After proof of the above theorem, then some examples will be given to show that the balanced
cone can be bigger than the image of the K\"ahler cone under the map $\mathbf b$. We believe that
it will be very interesting
if one can describe $\bf \overline b (\overline{\mathcal{K}})\cap \mathcal{B}$ clearly for a compact K\"ahler manifold.
\vspace{2mm}

In the last part we assume that $X$ is an $n$-dimensional K\"ahler manifold with holomorphically
trivial canonical bundle. (Hence, $X$
is a Calabi-Yau manifold.)
We will give an analytic method to distinguish the K\"ahler classes
and the nef but not K\"ahler classes which are mapped into
the balanced cone. We fix a Calabi-Yau metric $\omega_0$ satisfying
$\int_X \omega_0^n=1$ and a non-vanishing holomorphic $n$-form
$\zeta$ such that $\parallel\zeta\parallel_{\omega_0}=1$. For any
K\"ahler class $[\omega]$, Yau's theorem states that there exists a
unique Calabi-Yau metric $\omega_{CY}\in[\omega]$ such that
$\parallel\zeta\parallel_{\omega_{CY}}$, the (pointwise) norm of
$\zeta$ with respect to $\omega_{CY}$, is a constant. Under the
above assumption, this constant can be computed as follows:
\begin{equation*}
\parallel\zeta\parallel_{\omega_{CY}}^2=
\frac{\parallel\zeta\parallel_{\omega_{CY}}^2}{\parallel\zeta\parallel_{\omega_0}^2}=\frac{\omega_0^n}{\omega^n_{CY}}
=\frac{1}{\int_X\omega_{CY}^n}=\frac 1 {\int_X\omega^n}.
\end{equation*}
We can also ask whether there exists a balanced metric
$\Omega_{CY}$ in a given  balanced class $[\Omega]\in\mathcal B$ such
that
\begin{equation}\label{eon}
\parallel\zeta\parallel^2_{\Omega_{CY}}=c
\end{equation}
is a constant. This is the motivation of  papers
\cite{FWW10,FWW11}. There may be infinitely many solutions to equation
(\ref{eon}) in a given  balanced class. For example, Wang, Wu and
the first named author \cite{FWW10} proved that if $X$ is a
complex $n$-torus, then for a given  K\"ahler metric
$\omega$ and for any constant $c\geq (\int_{T^n}\omega^n)^{-1}$,
equation (\ref{eon}) has solutions in $[\omega^{n-1}]$.   They also
proved that for any Calabi-Yau manifold $X$ and a given K\"ahler
metric $\omega$ on $X$, if equation (\ref{eon}) has a solution in
$[\omega^{n-1}]$ for $c\leq (\int_X\omega^n)^{-1}$, then $c=(\int_X\omega^n)^{-1}$ and this
solution must be  the Calabi-Yau metric. Here we can prove that if
$\alpha$ is a nef but not K\"ahler class and $[\alpha^{n-1}]\in
\mathcal B$, then there  exists no solution in $[\alpha^{n-1}]$
of the  equation (\ref{eon}) for $c\leq(\int_X\alpha^n)^{-1}$.

\begin{theo}\label{chara}
Let $X$ be an $n$-dimensional Calabi-Yau manifold with a Calabi-Yau
metric $\omega_0$ and a non-vanishing holomorphic $n$-form $\zeta$
such that $\int_X\omega_0^n=1$ and
$\parallel\zeta\parallel_{\omega_0}=1$. Let
$[\alpha]\in\overline{{\mathcal K}}$ such that
$\overline{\mathbf b}([\alpha])=[\alpha^{n-1}]\in \mathcal B$.
\begin{enumerate}
\item If $[\alpha]\in\partial\mathcal K$, then equation (\ref{eon}) for $c\leq (\int_X\alpha^n)^{-1}$ has no
solution in the balanced class $[\alpha^{n-1}]$.
\item  If $[\alpha]\in \mathcal K$,
then there exists a unique solution $\Omega_{CY}\in[\alpha^{n-1}]$ of the equation (\ref{eon})
for $c\leq (\int_X\alpha^n)^{-1}$.  Actually in this case,
$c= (\int_X\alpha^n)^{-1}$
and $\Omega_{CY}=\omega_{CY}^{n-1}$
for the unique Calabi-Yau metric $\omega_{CY}$ in the K\"ahler class $[\alpha]$.
\end{enumerate}
\end{theo}

It is conjectured that for any $c>(\int_X\alpha^n)^{-1}$, the form-type Calabi-Yau equation (\ref{eon}) has solutions in the
balanced class $[\alpha^{n-1}]$ in the above theorem.\footnote{Recently, this conjecture has been solved by V. Tosatti and B. Weinkove in \cite{TW13}.} \\

The paper is organized as follows. In Section 2 we prove Proposition \ref{injective map} and Theorem \ref{inj_thm}, which will also be
 generalized to the Fujiki class $\mathcal C$. In Section 3 we prove Theorem \ref{nefb} and
give two examples.
In Section 4 we will use Theorem \ref{inj_thm} to prove Theorem \ref{chara}. Finally, for reader's convenience, we show in the appendix the equivalence of the balanced cone and the movable cone of a K\"ahler manifold under the assumption $\mathcal{E}^ \vee =\overline{\mathcal{M}}$ following the arguments of Toma.

{\bf Acknowledgments.} We would like to thank Prof. J.-P. Demailly, Zhizhang Wang, D. Wu and Prof. S.-T. Yau
for useful discussions and V. Tosatti for useful suggestions and comments. We are also indebted to the referees for helpful comments and suggestions.
Fu is supported in part by NSFC grants 11025103 and 11121101.

\section{Injectivity}
In this section, as a warm-up we first prove Proposition \ref{injective map}, which states that
the map $\mathbf b$ is injective. We remark that this is just a special case of Theorem \ref{inj_thm}. By presenting its proof here, we want to emphasize how to apply the solutions of the  complex Monge-Amp\`{e}re equations and the AM-GM inequality to obtain the result.
\begin{proof}
We need to prove that if $\omega_1$ and $\omega_2$ are two K\"ahler
metrics on $X$ satisfying
\begin{equation}\label{relation}
\omega^{n-1}_1=\omega_2^{n-1}+i\partial\bar\partial\varphi
\end{equation}
for some real $(n-2,n-2)$-form $\varphi$, then there exists a smooth
function $f$ on $X$ such that
$$\omega_1=\omega_2+i\partial\bar\partial f.$$

Let us first recall Yau's theorem on the complex
Monge-Amp\`ere equations on a compact K\"ahler manifold.
\begin{lemm}\textup{(}\cite{Yau78}\textup{)}\label{yau}
Let $X$ be a compact $n$-dimensional K\"ahler manifold with a
K\"ahler metric $\omega$. Then for any smooth volume form $\eta>0$
satisfying $\int_X\eta=\int_X\omega^n$, there exists a unique K\"ahler
metric $\tilde\omega=\omega+i\partial\bar\partial u$ in the
K\"ahler class $[\omega]$ such that $\tilde\omega^n=\eta$.
\end{lemm}
We use Yau's theorem as follows. Let $c$ be the following constant:
$$\int_X\omega^n_2=c\int_X\omega^n_1.$$
Without loss of generality, we  assume that $c\geq 1$. Since the
class $[\omega_2]$ is K\"ahler, by Yau's theorem we can
find a representative $\tilde\omega_2=\omega_2+i\partial\bar\partial
u$ of $[\omega_2]$ such that
\begin{equation}\label{201}
\tilde\omega_{2}^n=c\omega_1^n.
\end{equation}
However, the equalities
 $$[\tilde\omega_{2}^{n-1}]=[\omega_2^{n-1}]=[\omega_1^{n-1}]$$
imply that there exists a real $(n-2,n-2)$-form $\phi$ such that
\begin{equation}\label{202}
\tilde\omega_{2}^{n-1}=\omega_1^{n-1}+i\partial\bar\partial\phi.
\end{equation}
We will use the following
notations (see \cite{FWW10}). If $\Theta$ is a real
$(n-1,n-1)$-form, then $(\Theta_{i\bar j})$ is the matrix whose entries are
the coefficients of $\Theta$, and $(\Theta^{i\bar j})$ is its inverse
matrix. We will also denote $\det\Theta=\det(\Theta_{i\bar j})$.
Hence, combining (\ref{201}) with (\ref{202}), we find
\begin{eqnarray*}
\begin{aligned}
\Bigl(\frac 1
c\Bigr)^{n-1}=&\Bigl(\frac{\det\omega_1}{\det\tilde\omega_{2}}\Bigr)^{n-1}
=\frac{\det\omega_1^{n-1}}{\det\tilde\omega_{2}^{n-1}}\\
=&\frac{\det\omega_1^{n-1}}{\det(\omega_1^{n-1}+i\partial\bar\partial\phi)}.
\end{aligned}
\end{eqnarray*}

Now we follow the proof of Lemma 10 in \cite{FWW10}.
 We apply the
AM-GM inequality to obtain
  \begin{eqnarray} \label{eq:agm}
  \begin{aligned}
     c^{\frac  {n-1}{n}}
     =&\Bigl(\frac{\det(\omega_1^{n-1}+i\partial\bar\partial\phi)}{\det\omega_1^{n-1}}\Bigr)^{\frac
     1 n} \\
     \le & 1 + \frac{1}{n} \sum_{i,j} (\omega_1^{n-1})^{i\bar{j}} \Big(i\partial\bar\partial \phi\Big)_{i\bar{j}},
     \end{aligned}
  \end{eqnarray}
  which implies
  $$ c^{\frac {n-1}{n}}\omega_1^n\leq \omega_1^n+i\partial\bar\partial\phi\wedge\omega_1.$$
Integrating over $X$, since $\omega_1$ is K\"ahler, we get
  \[
     c^{\frac {n-1} n}\int_X \omega_1^n \le  \int_X \omega_1^n. \]
This shows that $c=1$ and  a pointwise equality in
\eqref{eq:agm} holds. This forces that $i\partial\bar\partial \phi = 0$.
Therefore,  (\ref{202}) implies  $\tilde\omega_{2}=\omega_1$ and
$\omega_1=\omega_2+i\partial\bar\partial u$.
\end{proof}

%\begin{rema}
%The equality $c=1$ can also be derived by Demailly or Gromov's convexity inequality, see \cite{Dem93} or \cite{Gro90}. \textup{(}Or see the proof of the next theorem.\textup{)}
%\end{rema}

\begin{rema}
\textup{
When $n=3$, equation (\ref{relation}) implies
$$(\omega_1-\omega_2)\wedge (\omega_1+\omega_2)=i\partial\bar\partial\varphi.$$
Since $\omega_1+\omega_2$ is a K\"ahler metric, by the hard
Lefschetz theorem, $\omega_1+\omega_2$ defines an isomorphism from
$H^{1,1}_{dR}(X,\mathbb R)$ to $H^{2,2}_{dR}(X,\mathbb R)$. Hence
$\omega_1-\omega_2$ is trivial in $H^{1,1}_{dR}(X,\mathbb R)$. For
$n>3$, we can rewrite  (\ref{relation}) as
\begin{equation*}
(\omega_1-\omega_2)\wedge\Bigl(\sum_{k=0}^{n-2}\omega_1^{n-k-2}\wedge\omega_2^k\Bigr)=i\partial\bar\partial\varphi.
\end{equation*}
Here $\sum_{k=0}^{n-2}\omega_1^{n-k-2}\wedge\omega_2^k$ is a
$d$-closed strictly positive definite $(n-2,n-2)$-form. In general,
such a form cannot be represented by $\omega_0^{n-2}$ for some
Hermitian metric $\omega_0$. Otherwise $\omega_0$ is also K\"ahler
\textup{(}cf. \cite{GH}\textup{)} and then the hard Lefschetz theorem
also implies that $\omega_1-\omega_2$ is trivial. Anyway, we don't know
whether there exists an algebraic proof of Proposition 1.1.}
\end{rema}

We can generalize the above proposition from the K\"ahler classes to the  nef
and big classes. Instead of constructing two equal K\"ahler metrics, we will construct two equal currents. Hence, we need the following important theorem in \cite{BEGZ}.
\begin{lemm}\textup{(\cite{BEGZ})}
\label{begz10} Let $X$ be a compact $n$-dimensional K\"ahler
manifold and let $\eta$ be a smooth volume form on $X$. Let
$[\alpha]$ be a nef and big class on $X$. Then there exists a unique
$\alpha$-psh function $u$ with $\sup_{X}u=0$ such that
$$\bigl\langle(\alpha+i\partial\bar{\partial}u)^{n}\bigr\rangle=c\eta \ \ \ \textup{with}\
c=\frac{\int_{X}\alpha^{n}}{\int_{X}\eta}>0.$$ Here  $\langle\cdot\rangle$
denotes the non-pluripolar product of positive currents. Moreover,
$u$ has minimal singularities and is smooth on
$\textup{Amp}(\alpha)$, which is a Zariski open set of $X$ depending
only on the cohomology class of $\alpha$.
\end{lemm}
Recall that $u$ is called an $\alpha$-psh function if $\alpha+i\partial \bar \partial u$ is a positive current. Let us briefly discuss how the result is obtained. In fact, by Yau's theorem, the
above degenerate complex Monge-Amp\`ere equation can be solved by
approximation.  Fix a K\"ahler metric $\omega$ on X.
If we write $c_{t}=\int_{X}(\alpha+t\omega)^{n}/\int_{X}\eta$ with
$0<t<1$, then there exists a unique smooth function $u_t$ with
$\sup_{X}u_t=0$ such that
$$(\alpha+t\omega+i\partial\bar{\partial}u_t)^{n}=c_{t}\eta. $$
First, by basic properties of plurisubharmonic functions, the
family of solutions $u_t$ is compact in $L^1(X)$-topology and then
there exists a sequence $u_{t_k}$ such that
\begin{equation*}
\alpha+t_k\omega+i\partial\bar{\partial}u_{t_k}\to
\alpha+i\partial\bar{\partial}u \ \ \ \ \ \textup{as currents on} \
X.
\end{equation*}
Moreover, by the theory developed in \cite{BEGZ} and Yau's basic estimates
in \cite{Yau78}, $u_t$ is compact in
$C^{\infty}_{\textup{loc}}(\textup{Amp}(\alpha))$. Therefore there exists a
subsequence of $u_{t_k}$, which is still denoted as $u_{t_k}$, (we will not stress this point in the following,)  such that
\begin{equation*}
\alpha+t_k\omega+i\partial\bar{\partial}u_{t_k}\to
\alpha+i\partial\bar{\partial}u \ \ \ \ \ \textup{in}\ \
C^{\infty}_{\textup{loc}}(\textup{Amp}(\alpha)).
\end{equation*}
Hence $u$ is smooth on $\textup{Amp}(\alpha)$. Since $\eta$ is the
smooth volume form, $\alpha+i\partial\bar{\partial}u$ is a K\"ahler
metric on $\textup{Amp}(\alpha)$.

Now we are ready to prove Theorem \ref{inj_thm}. We rephrase it as
\begin{theo}\label{bignef}
Let $X$ be a compact $n$-dimensional K\"ahler manifold. If
$[\alpha]$ and $[\beta]$ are two nef and big classes and
$[\alpha^{n-1}]=[\beta^{n-1}]$, then $[\alpha]=[\beta]$.
\end{theo}
\begin{proof}
Since $\alpha$ and $\beta$ are nef and
$[\alpha^{n-1}]=[\beta^{n-1}]$, we have
\begin{equation}\label{star}
\int_X\beta^n=\int_X\beta\wedge\alpha^{n-1}.
\end{equation}
Then by the convexity inequality in
\cite{Dem93} or \cite{Gro90}, we have
\begin{eqnarray*}
\label{volume ineq}
  \int_{X}\beta^{n}\geq \Bigl(\int_{X}\beta^{n}\Bigr)^{\frac{1}{n}}
  \Bigl(\int_{X}\alpha^{n}\Bigr)^{\frac{n-1}{n}},
\end{eqnarray*}
which implies $\int_{X}\beta^{n}\geq
\int_{X}\alpha^{n}$. Similarly we also have
$\int_{X}\alpha^{n}\geq \int_{X}\beta^{n}$. Thus we get
\begin{equation}\label{volume eq}
    \int_{X}\alpha^{n}= \int_{X}\beta^{n}.
\end{equation}

We fix a K\"ahler metric $\omega$ and a volume form $\eta$ on $X$.
We denote for  $0<t<1$
$$c=\frac{\int_{X}\alpha^{n}}{\int_{X}\eta},\ \ c_{\alpha,t}=\frac{\int_{X}(\alpha+t\omega)^{n}}{\int_{X}\eta}, \ \ \textup{and}\ \
c_{\beta,t}=\frac{\int_{X}(\beta+t\omega)^{n}}{\int_{X}\eta}.$$
Then Lemma \ref{yau} implies that there exist two families of
smooth functions $u_{t}$ and $v_{t}$ such that, if we denote  $\alpha_t=\alpha+t\omega+i\partial\bar\partial u_t$
and $\beta_t=\beta+t\omega+i\partial\bar\partial v_t$, then
\begin{eqnarray*}
   \alpha_t^{n}=c_{\alpha,t}\eta
   \ \ \ \ \ \ \textup{and}\ \ \ \ \ \
    \beta_t^{n}=c_{\beta,t}\eta.
\end{eqnarray*}
Hence
\begin{equation}\label{two app ratio}
   \frac{\alpha_t^{n}}{\beta_t^{n}}=c_t
\end{equation}
with
\begin{equation}
c_{t}=\frac{c_{\alpha,t}}{c_{\beta,t}}=\frac{\int_{X}(\alpha+t\omega)^{n}}{\int_{X}(\beta+t\omega)^{n}}.
\end{equation}
Then identity (\ref{volume eq}) implies
\begin{equation}\label{limitc}
\lim_{t\to 0}c_{t}=1.
\end{equation}

By the assumption $[\alpha^{n-1}]=[\beta^{n-1}]$, there exists a
$(n-2,n-2)$-form $\phi$ such that
$\alpha^{n-1}=\beta^{n-1}+i\partial\bar{\partial}\phi$. We rewrite it as
\begin{equation}\label{app n-1}
    \alpha_{t}^{n-1}=\beta_{t}^{n-1}+\Theta_{t}
\end{equation}
where
 $$\Theta_{t}=i\partial\bar{\partial}\phi+
\sum_{k=0}^{n-2}C_{n-1}^k\bigl(\alpha^k\wedge(t\omega+i\partial\bar\partial
u_t)^{n-1-k}-\beta^k\wedge (t\omega+i\partial\bar\partial
v_t)^{n-1-k}\bigr).$$
Then applying the AM-GM
inequality to (\ref{two app ratio}), we have
\begin{equation}\label{qqq}
\begin{aligned}
   c_{t}^{\frac{n-1}{n}}=&\Bigl(\frac{\det\alpha_{t}^{n-1}}{\det\beta_{t}^{n-1}}\Bigr)^{\frac 1 n}=
   \Bigl(\frac{\det(\beta_{t}^{n-1}+\Theta_{t})}{\det\beta_{t}^{n-1}}\Bigr)^{\frac 1 {n}}\\
 &\leq 1+\frac{1}{n} \sum_{i,j}(\beta_{t}^{n-1})^{i\bar{j}}(\Theta_{t})_{i\bar{j}}.
 \end{aligned}
\end{equation}
We multiply the volume form $\beta_{t}^{n}$ to both sides of the
above inequality and get
\begin{equation}\label{arith-geom ineq form}
    c_{t}^{\frac{n-1}{n}}\beta_{t}^{n}\leq\beta_{t}^{n}+\beta_{t}\wedge\Theta_{t}.
\end{equation}

Next, we consider the limit of $\beta_t\wedge\Theta_t$ as $t$ goes
to zero. By (\ref{app n-1}), we have
\begin{equation}\label{01}
\beta_t\wedge \Theta_t=\beta_t\wedge\alpha_t^{n-1}-\beta_t^n.
\end{equation}
It is easy to see the positive measures  $\beta_{t}\wedge\alpha_{t}^{n-1}$ and
$\beta_{t}^{n}$ have uniformly bounded masses:
\begin{eqnarray*}
\begin{aligned}
||\beta_{t}\wedge\alpha_{t}^{n-1}||_{\textup{mass}}=&\int_{X}(\beta+t\omega)\wedge(\alpha+t\omega)^{n-1}\\
< &\int_{X}(\beta+\omega)\wedge(\alpha+\omega)^{n-1},
\end{aligned}
\end{eqnarray*}
 and
$$||\beta_{t}^{n}||_{\textup{mass}}=\int_{X}(\beta+t\omega)^{n}<\int_{X}(\beta+\omega)^{n}.$$
Hence we can pick a decreasing subsequence $t_k\to 0$ such that
$\beta_{t_k}\wedge\alpha_{t_k}^{n-1}$ and $\beta_{t_k}^n$ converge weakly to $\mu_1$ and $\mu_2$ respectively. Therefore
if we denote $\mu_0=\mu_1-\mu_2$, then as currents,
\begin{equation*}\label{lim measure}
  \beta_{t_k}\wedge\Theta_{t_k}\to\mu_{0} \ \ \ \ \textup{when}\ \ \
  \ t_k\to 0.
\end{equation*}
Moreover, it is not hard to see from (\ref{limitc}) and
(\ref{arith-geom ineq form}) that $\mu_0$ is a positive measure on
$X$. Meanwhile, by (\ref{01}) and (\ref{star}),
\begin{equation*}
\begin{aligned}
    \int_{X}\mu_{0}
    =&\lim_{t_k\to 0}\int_{X}(\beta_{t_k}\wedge\alpha_{t_k}^{n-1}-\beta_{t_k}^{n})
    =&\int_{X}(\beta\wedge\alpha^{n-1}-\beta^{n})=0.
    \end{aligned}
\end{equation*}
Hence $\mu_{0}$ is  a zero measure. In particular,  since
$\Xi:=\textup{Amp}(\alpha)\cap \textup{Amp}(\beta)$ is a Borel set,
we have
\begin{equation}\label{app lim measure}
  \beta_{t_k}\wedge\Theta_{t_k}\to 0 \ \ \ \ \textup{as currents on
  $\Xi$}.
\end{equation}

On the other hand, by Lemma \ref{begz10}, there exists a unique
$\alpha$-psh function $u_0$ with $\sup_X u_0=0$ and a unique
$\beta$-psh function $v_0$ with $\sup_X v_0=0$ such that
$u_0$ (resp. $v_0$) is smooth on $\textup{Amp}(\alpha)$ (resp. $\textup{Amp}(\beta)$) and
\begin{eqnarray*}
  \bigl\langle(\alpha+i\partial\bar{\partial}u_0)^{n}\bigr\rangle=c\eta,\ \ \ \ \
  \bigl\langle(\beta+i\partial\bar{\partial}v_0)^{n}\bigr\rangle=c\eta,
\end{eqnarray*}
Here by (\ref{volume eq}), we have $$c=\frac{\int_X\alpha^n}{\int_X\eta}=\frac{\int_X\beta^n}{\int_X\eta}.$$ If we denote
$\alpha_0=\alpha+i\partial\bar\partial u_0$ and
$\beta_0=\beta+i\partial\bar\partial v_0$, then as discussed before,
there exist subsequences $\alpha_{t_k}$ of $\alpha_t$ and
$\beta_{t_k}$ of $\beta_t$ such that
\begin{equation*}
\alpha_{t_k}\to \alpha_{0} \ \ \ \ \textup{in}\
C^\infty_{\textup{loc}}(\textup{Amp}(\alpha))
\end{equation*}
and
\begin{equation*}
\beta_{t_k}\to \beta_{0} \ \ \ \ \textup{in}\
C^\infty_{\textup{loc}}(\textup{Amp}(\beta)).
\end{equation*}
Thus,
\begin{equation}\label{smooth lim1}
   \Theta_{t_k}\to\Theta_0\ \ \ \textup{and}\ \ \ \beta_{t_k}\wedge\Theta_{t_k}\to\beta_{0}\wedge\Theta_{0} \ \ \ \ \textup{in} \  C^\infty_{\textup{loc}}(\Xi)
\end{equation}
for some smooth form  $\Theta_{0}$ which is only defined  on $\Xi$.
Combining (\ref{app lim measure}) with (\ref{smooth lim1}) and
using uniqueness of the limit, we obtain
\begin{equation*}\label{zero}
    \beta_{0}\wedge\Theta_{0}=0 \ \ \ \ \text{on} \  \Xi.
\end{equation*}
The above equality and (\ref{limitc}) imply that on $\Xi$, if we take the limits of both side of (\ref{qqq}) as $t\to 0 $,
\begin{eqnarray*}
 \label{arith-geom ineq zero}
   1&=&\bigl(\frac{\det\alpha_{0}^{n-1}}{\det\beta_{0}^{n-1}}\bigr)^{\frac{1}{n}}=
   \bigl(\frac{\det(\beta_{0}^{n-1}+\Theta_{0})}{\det\beta_{0}^{n-1}}\bigr)^{\frac{1}{n}} \\
 &\leq& 1+\frac{1}{n}
 \sum_{i,j}(\beta_{0}^{n-1})^{i\bar{j}}(\Theta_{0})_{i\bar{j}}=1,
\end{eqnarray*}
which forces $\Theta_0=0$ on $\Xi$.  Hence $\alpha_0^{n-1}=\beta_0^{n-1}$ on
$\Xi$. Since $\alpha_0$ and $\beta_0$ are K\"ahler metrics on $\Xi$,
we have $\alpha_0=\beta_0$ on $\Xi$.

We claim $\alpha_0=\beta_0$ on $X$. First, we need the following two lemmas.
\begin{lemm}\label{current spt}\textup{(}\cite{DEL}\textup{)}
Let $T$ be a $d$-closed $(p,p)$-current and $\textup{supp}\hspace{0.5mm}  T$  contained in an analytic subset $A$.
If $\dim A<n-p$, then $T=0$; if $T$ is of order zero and $A$ is of dimension $n-p$ with $(n-p)$-dimensional
irreducible components $A_{1},\cdots,A_{k}$, then $T=\sum c_{j}[A_{j}]$ with $c_{j}\in \mathbb{C}$.
\end{lemm}

\begin{lemm}\label{lelong num zero}\textup{(}\cite{B04}\textup{)}
Let $[\alpha]$ be a nef and big class, and let $T_{\min}$ be a positive current in $[\alpha]$ with minimal singularities.
Then the Lelong number $\nu(T_{\min},x)=0$ for any point $x\in X$.
\end{lemm}

Now if we let $T=\alpha_0-\beta_0$, then $T$ is a real $d$-closed (1,1)-current and $\textup{supp}\hspace{0.5mm} T\subset X-\Xi$.
If $X-\Xi$ is of codimension more than one,
then the first part of Lemma \ref{current spt} implies $T=0$. Hence $\alpha_{0}=\beta_{0}$ on $X$. If
$X-\Xi$ has irreducible components $D_{1},\cdots,D_{k}$ of pure codimension one,
then the second part of Lemma \ref{current spt} implies $\alpha_{0}-\beta_{0}=\sum c_{j}[D_{j}]$. We should also consider the following most complicated case:  $X-\Xi$ has irreducible components of codimension one, whose union is denoted  by $D$, and also has of codimension greater than one, whose union is denoted by $F$.  In this case, we use the same argument of the proof of the second part of Lemma \ref{current spt} (cf. page 143 of \cite{DEL}). The regular part  $D_{reg}$ of $D$ is a complex submanifold of $X-(D_{sing}\cup F)$, where $D_{sing}$ is the singular part of $D$,  and its connected components are
$D_j \cap D_{reg}$. Then we apply the second theorem of support (see page 142 of \cite{DEL}) to get $\alpha_{0}-\beta_{0}=\sum c_{j}[D_{j}]$ on $X-(D_{sing}\cup F)$. Now $\alpha_{0}-\beta_{0}-\sum c_{j}[D_{j}]$ is a $d$-closed current of order $0$ and its support is contained in $D_{sing}\cup F$ of codimension greater than one. So the current $\alpha_{0}-\beta_{0}-\sum c_{j}[D_{j}]$ must vanish by the first part of Lemma \ref{current spt}. Hence, for the last two cases, we should prove $c_j=0$ for any $j$.

Since $\alpha_0$ and $\beta_0$ are real, all
$c_j$'s can be chosen to be real.
If there exists at least one $c_j>0$, we can write this equality as
\begin{equation}\label{current eq}
    \alpha_{0}-\sum c_{j'}[D_{j'}]=\beta_{0}+\sum c_{j''}[D_{j''}]
\end{equation}
with $c_{j'}\leq 0$ and $c_{j''}>0$.
Fix one such $j''$, which we denote as $j_0''$. We take a generic point $x\in D_{j''_0}$, for example,
we can take such a point $x$ with $\nu([D_{j''_0}],x)=1$ and $x\notin X-D_{j''_0}$.
Then taking the Lelong number at the point $x$ on  both sides of (\ref{current eq}), we find
\begin{equation*}\label{lelong eq}
    \nu(\alpha_{0},x)-\sum c_{j'}\nu([D_{j'}],x)=\nu(\beta_{0},x)+\sum c_{j''}\nu([D_{j''}],x).
\end{equation*}
Since $\alpha_{0}$ and $\beta_{0}$ are positive currents with minimal singularities
in nef and big classes, Lemma \ref{lelong num zero} tells us that $\nu(\alpha_{0},x)=0$ and $\nu(\beta_0,x)=0$.
The property of $x$ also implies $\nu([D_{j'}],x)=0$ and $\nu([D_{j''}],x)=0$ for all $j'$ and all $j''\not=j''_0$.
All these force $c_{j_0''}=0$,
which contradicts our assumption that $c_{j_0''}>0$.  Thus we have
\begin{equation*}
    \alpha_{0}-\sum c_{j'}[D_{j'}]=\beta_{0}.
\end{equation*}
By the same argument, we can also prove $c_{j'}=0$. Hence $\alpha_0=\beta_0$ on $X$.
Therefore, we have $[\alpha]=[\beta]$ on $X$.
\end{proof}

The above result is also valid if $X$ is merely in the Fujiki class $\mathcal C$. For a general compact complex manifold, a cohomology class $[\alpha]_{bc} \in H_{BC}^{1,1}(X, \mathbb{R})$ is called nef if for any $\varepsilon >0$, there exists a smooth function $\psi_\varepsilon$ such that $\alpha+i\partial \bar \partial \psi_\varepsilon > -\varepsilon \omega$.
\begin{coro}
Let $X$ be a compact complex $n$-dimensional manifold in the Fujiki
class $\mathcal C$. If $[\alpha]$ and $[\beta]$ are
two nef and big classes, and  $[\alpha^{n-1}]=[\beta^{n-1}]$, then  $[\alpha]=[\beta]$.
\end{coro}
\begin{proof}
Since $X$ is in the Fujiki class $\mathcal C$, there exists a proper modification
$\mu:\bar{X}\rightarrow X$ with $\bar{X}$ a compact K\"ahler manifold.
By assumptions on $\alpha$ and $\beta$, $[\mu^{*}\alpha]$ and $[\mu^{*}\beta]$ are also
nef and big classes on $\bar{X}$, satisfying
$$[(\mu^{*}\alpha)^{n-1}]=[(\mu^{*}\beta)^{n-1}].$$
Then by the theorem above, we have
$$[\mu^{*}\alpha]=[\mu^{*}\beta].$$
As $\mu$ is a proper modification, this implies that
 $[\alpha]=[\beta]$ on $X$.
\end{proof}

Note that on a Moishezon manifold, M. Paun (\cite{P98})  proved that
for a holomorphic line bundle $L$, $c_{1}(L)$ is nef
if and only if $L\cdot C\geq 0$ for every irreducible curve $C$. Thus our result yields the following
\begin{coro}
Let $X$ be a compact $n$-dimensional Moishezon manifold. Let $L$ be a big line bundle over $X$ and
$L\cdot C\geq 0$ for every irreducible curve $C$ on $X$. Then $c_1(L)^{n-1}$ determines $c_1(L)$.
\end{coro}

\section{The image of the boundary of the K\"ahler cone}
Sometimes it is convenient to consider the Aeppli cohomology groups
$V^{p,q}(X,\mathbb{C})$. Since we are interested in the real case,
 we give the following
\begin{defi}
\label{bca} If we denote by $A^{p,q}(X)$ the space of the smooth $\mathbb C$-valued $(p,q)$ forms and by $A^{p,p}_{\mathbb{R}}(X)$ the space of the
smooth $\mathbb {R}$-valued $(p,p)$-forms, then
$$V^{p,p}(X,\mathbb{R})
=\frac{\{\phi\in
A^{p,p}_{\mathbb{R}}(X)|\partial\bar\partial\phi=0\}}{\{\partial{A^{p-1,p}(X)}
+{\bar{\partial}}A^{p,p-1}(X) \}\cap A^{p,p}_{\mathbb R}(X)}.$$
\end{defi}
We denote the space of $(p,q)$-currents by ${D}'^{p,q}(X)$.
Then it is well known that we can also replace $A^{p,q}$ by
$D'^{p,q}$ in the above definition. We denote an element of the
above cohomology groups by $[\cdot]_{a}$.

We need the following lemma due to Bigolin.
\begin{lemm}
\label{apelli} \textup{(}\cite{B69}\textup{)} Let $X$ be a compact
complex $n$-dimensional manifold. The dual space of the $(p,p)$-th Aeppli group is
just the $(n-p,n-p)$-th Bott-Chern group, i.e.,
\begin{eqnarray*}
V^{p,p}(X,{\mathbb{R}}){'}=H^{n-p,n-p}_{BC}(X,{\mathbb{R}}).
\end{eqnarray*}
\end{lemm}

In particular, $V^{p,p}(X,{\mathbb{R}})$ is a finite dimensional
vector space. Furthermore, if $X$ satisfies the
$\partial\bar\partial$-lemma, then $\dim
V^{p,p}(X,{\mathbb{R}})=h^{p,p}$, where $h^{p,p}$ is the Hodge
number. The following lemma is inspired by the method in \cite{MT09} and \cite{Sul76}. In fact,
it is an easy consequence of the Hahn-Banach theorem.
\begin{lemm}
\label{dbclemma} Let $X$ be a compact complex $n$-dimensional
manifold. Suppose that $\Omega_0$ is a real $d$-closed
$(n-1,n-1)$-form satisfying that, for any positive
$\partial\bar\partial$-closed $(1,1)$-current $T$,
$\int_{X}\Omega_0\wedge T\ge 0$ and $\int_{X}\Omega_0\wedge T= 0$ if and
only if $T=0$. Then $[\Omega_0]$ is a balanced class.
\end{lemm}
\begin{proof}
Fix a Hermitian metric $\omega$ on $X$. We define the following two
subsets of $D'^{1,1}_{\mathbb{R}}(X)$ :
\begin{align*}
D_{1}&=\{T\in D'^{1,1}_{\mathbb{R}}(X)|\partial\bar\partial T=0 , \int_{X}\Omega_0\wedge T=0 \},\\
D_{2}&=\{T\in D'^{1,1}_{\mathbb{R}}(X)|T\ge0 ,
\int_{X}\omega^{n-1}\wedge T=1 \}.
\end{align*}
Then $D_{1}$ is a closed subspace and $D_{2}$ is a compact convex
subset under the weak topology of currents. Since $\Omega_0$ is
$d$-closed,
\begin{equation}\label{0011}
\{\partial\bar {S}
+\bar{\partial} S|S\in D'^{1,0}(X,{\mathbb{C}})\}\subset D_1\ .
\end{equation}
It is clear that
$D_{1}\cap D_{2}$ is empty. By the Hahn-Banach theorem, there exists
a smooth real $(n-1,n-1)$-form $\Omega$ such that
\begin{equation}\label{0010}
\Omega|_{D_{1}}=0\qquad  \textup{and}\qquad  \Omega|_{D_{2}}>0\ .
\end{equation}
The identity in (\ref{0010}) and (\ref{0011})  imply $d\Omega=0$, and the inequality in (\ref{0010})
implies that $\Omega$ is strictly positive. Hence $\Omega$ is a balanced
metric.

On the other hand, Lemma \ref{apelli} says that $[\Omega_0]$ and
$[\Omega]$ are linear functionals on $V^{1,1}(X,\mathbb{R})$. We
have a natural  projective map
\begin{align*}
  \pi:\{T\in D'^{1,1}_{\mathbb{R}}(X)|\partial\bar\partial T=0 \}\to
  V^{1,1}(X,\mathbb{R})
\end{align*}
with $\pi(T)=[T]_a$.
Then the definition of $D_{1}$ implies $\pi(D_{1})=\ker[\Omega_0]$, and
$\Omega|_{D_{1}}=0$ implies $\pi(D_{1})\subseteq \ker[\Omega]$. Thus
we have
\begin{align*}
  \ker[\Omega_0]\subseteq \ker[\Omega]\subseteq
  V^{1,1}(X,\mathbb{R}).
\end{align*}
If $\ker[\Omega]$ is the whole Aeppli group,
then $[\Omega]$=0. Since $X$ is compact, there exists an
$\varepsilon>0$ small enough such that $\Omega+\varepsilon\Omega_0>0$, i.e.,
$[\Omega]+\varepsilon[\Omega_0]=\varepsilon[\Omega_0]$ is balanced.
If $\ker[\Omega]$ is a proper subspace, since
$V^{1,1}(X,\mathbb{R})$ is a finite dimensional vector space, we
must have $\ker[\Omega_0]=\ker[\Omega]$. Hence there exists some
constant $c$ such that $[\Omega_0]=c[\Omega]$. In this case, if there exists some non-trivial positive $\partial \bar \partial$-closed $(1,1)$-current $T$,  the constant $c$ must be positive, and this implies that
$[\Omega_0]$ is balanced. Otherwise, if there is no
non-trivial positive $\partial \bar \partial$-closed $(1,1)$-current, then the zero class satisfies our assumption in the lemma and we can repeat our procedure above. We can use the zero class to define the space $D_1$. Hence the zero class is a balanced class. This means that every class in $H^{n-1,n-1}_{BC}(X, \mathbb{R})$ is balanced.
Thus we
finish the proof of Lemma \ref{dbclemma}.
\end{proof}

\begin{rema}
\label{balanced_dual_rmk}
Let $X$ be a compact balanced manifold. If we denote $\mathcal{E}_{dd^c}\subseteq V^{1,1}(X, \mathbb{R})$ the convex cone generated by $dd^c$-closed positive $(1,1)$-currents, then the above lemma implies $\mathcal{E}_{dd^c}^\vee =\overline{\mathcal{B}}$.
\end{rema}
The above lemma has as corollary the following two interesting propositions. Let $\Omega_0$ be a semi-positive $(n-1,n-1)$-form
on $X$ which is  strictly positive on $X-V$ for a subvariety $V$ of $X$. If $\textup{codim}\ V>1$, we first recall
Theorem 1.1 in \cite{AB92}.
\begin{lemm}\cite{AB92}
Let $X$ be a complex  $n$-dimensional  manifold. Assume $T$ is a $\partial\bar\partial$-closed positive
$(p,p)$-current on $X$
such that the Hausdorff
$2(n-p)$-measure of $\textup{supp}\hspace{0.5mm}  T$ vanishes. Then $T=0$.
\end{lemm}

\begin{prop}
\label{dbc prop 1} Let  $X$ be a compact complex $n$-dimensional
manifold. If $\Omega_0$ is a $d$-closed semi-positive $(n-1,n-1)$-form  on $X$
and  is strictly positive outside a subvariety $V$ with
$\textup{codim}\ V>1$, then $[\Omega_0]$ is a balanced class.
\end{prop}
\begin{proof}
 Fix a  $\partial\bar\partial$-closed positive
$(1,1)$-current $T$. Then $\Omega_0\geq 0$ implies that $\int_X\Omega_0\wedge T\geq 0$. And $\Omega_0>0$ on $X-V$
implies that  $\int_X\Omega_0\wedge T=0$ if and only if $\textup{supp}\hspace{0.5mm}  T\subset V$.
Hence according to the above lemma, since $T$ is $\partial\bar\partial$-closed and $\textup{codim}\ V>1$, we have $T=0$.
Thus $\Omega_0$
satisfies the conditions of Lemma \ref{dbclemma} and therefore is in the balanced class.
\end{proof}

If $\textup{codim}\ V=1$ and $\Omega_0$ is a balanced metric, we have $\int_V\Omega_0>0$.
We want to prove that this is also a sufficient
condition when $\Omega_0$ is semi-positive on $X$ and is strictly positive on $X-V$. We need Theorem 1.5 in \cite{AB92}.
\begin{lemm}
\label{ab92lemma} \textup{(\cite{AB92})} Let $X$ be a complex $n$-dimensional
manifold and $E$  a compact analytic subset. Let
$E_{1},\cdots,E_{k}$ be the irreducible $p$-dimensional components of $E$. If $T$ is a positive
$\partial\bar\partial$-closed $(n-p,n-p)$-current such that
$\textup{supp}\hspace{0.5mm} T\subset E$, then there exist constants $c_{j}\ge 0$
such that $T-\sum_{1}^{k}c_{j}[E_{j}]$ is a positive
$\partial\bar\partial$-closed $(n-p,n-p)$-current on $X$, supported  on the
union of the irreducible components of $E$ of dimension greater than
$p$.
\end{lemm}
Then we have
\begin{prop}
\label{dbc prop 2} Let $X$ be a compact complex $n$-dimensional
manifold. If $\Omega_0$ is a  $d$-closed semipositive $(n-1,n-1)$-form  on $X$
such that it is  strictly positive outside a codimension one
subvariety $V$ with irreducible components $E_{1},\cdots,E_{k}$ and
$[\Omega_0]\cdot[E_{j}]>0$ for $j=1,\cdots,k$, then $[\Omega_0]$ is a
balanced class.
\end{prop}
\begin{proof}
Since $\Omega_0$ is a semi-positive form on ${X}$, for
any $\partial\bar\partial$-closed positive $(1,1)$-current ${T}$ on ${X}$,
$\int_{{X}}\Omega_0\wedge{T}\ge 0$, and $\int_{{X}}\Omega_0\wedge{T}=0$
implies  $\textup{supp}\hspace{0.5mm}  T\subset V$. We need to prove $T=0$. By the above lemma, there exist
constants $c_{j}\ge 0$ such that
\begin{eqnarray*}
  T &=& \sum_{j=1}^{k}c_{j}[E_{j}].
\end{eqnarray*}
Hence $[\Omega_0]\cdot T=0$ implies that if  $[\Omega_0]\cdot[E_{j}]>0$, the constants $c_{j}$ must be zero.
 This implies $T=0$. Thus by Lemma \ref{dbclemma}, $[\Omega_0]$ is a balanced class.
\end{proof}

Before we  apply Proposition \ref{dbc prop 1}
%and Proposition \ref{dbc prop 2}
to
a nef and big class on a projective Calabi-Yau manifold, we need the following lemma given by Tosatti.
\begin{lemm}
\label{tos07} \textup{(}\cite{T07}\textup{)} Let $X$ be a projective Calabi-Yau
$n$-dimensional manifold and let $[\alpha]\in\partial\mathcal{K}_{NS}$ be a
big class. Then there exists a smooth form $\alpha_0\in[\alpha]$  which
is nonnegative and strictly positive outside a proper subvariety of
$X$.
\end{lemm}
For reader's convenience, we present some details on how to  prove the above lemma in \cite{T07}. First assume that $[\alpha] = c_1(L)$ for some holomorphic line bundle $L$, which
means that $[\alpha]$ lies in the space $NS(X)_{\mathbb{Z}}$. Hence, $L$ is nef and big. Now the base point free theorem implies that $L$ is semiample, so there exists some positive integer $k$
such that $kL$ is globally generated. This gives a holomorphic map $$F_{[\alpha]}: X\rightarrow \mathbb{P}(H^0 (X, \mathcal{O}(kL))^*)$$
such that $F_{[\alpha]} ^* \mathcal{O}(1) = kL$. If $[\alpha]\in NS(X)_{\mathbb{Q}}$, then $l[\alpha]\in NS(X)_{\mathbb{Z}}$ for some
positive integer $l$, and we can also define a holomorphic map $F_{[\alpha]}$ similarly.
Finally if $[\alpha]\in NS(X)_{\mathbb{R}}$, then
by Theorem 5.7 in \cite{Kaw88} or Theorem 1.9 in \cite{Kaw97}, we know that the subcone of nef and big classes is locally
rational polyhedral. Hence, $[\alpha]$ lies on a face of this cone which is cut out by
linear equations with rational coefficients. It follows that rational points on this face are dense, and it is then possible to write $[\alpha]$ as a linear combination
of classes in $NS(X)_{\mathbb{Q}}$ which are nef and big, with nonnegative coefficients.
Notice that all of these classes give the same contraction map,
because they lie on the same face. We also denote this map by $F_{[\alpha]}$.
Recall that the exceptional set $Exc(F_{[\alpha]})$ is defined to be the complement of points where $F_{[\alpha]}$ is a local isomorphism.
It is now clear that we can represent $\alpha$ by a smooth nonnegative form which is the pull back of Fubini-Study metric (up to scale). And it is strictly positive outside the exceptional set $Exc(F_{[\alpha]})$.

In birational geometry (cf. \cite{KMM87}), $F_{[\alpha]}$ is called a divisorial contraction if $Exc(F_{[\alpha]})$ is of codimension $1$ and a flipping contraction if the exceptional set $Exc(F_{[\alpha]})$ is of codimension greater than $1$.
We remark that if  $F_{[\alpha]}$ is a divisorial contraction, then the image of $Exc(F_{[\alpha]})$ under $F_{[\alpha]}$ is of dimension less than $n-1$. In our situation, $X$ is smooth, thus under divisorial contractions, its image is $Q$-factorial and has only weak log-terminal singularities (cf. Proposition 5-1-6 of \cite{KMM87}). Thus, its image is $Q$-factorial and normal. Then the image of $Exc(F_{[\alpha]})$ under $F_{[\alpha]}$ has codimension at least $2$ (cf. page 28 of \cite{OD01}). In this case,  $[\alpha^{n-1}]$ cannot be a balanced class.
Indeed, if $E_j$ is any codimension 1 component of $Exc(F_{[\alpha]})$, then we must have
$[\alpha^{n-1}]\cdot[E_j ] = 0$.
Write $Exc(F_{[\alpha]}) = F \cup_j E_j$ where all irreducible
components of $F$ have codimension at least $2$. For a fixed $j$ and for any
$p\in E_j \backslash (F\cup_{l\neq j} E_l)$, let
$S = F_{[\alpha]}^{-1}(F_{[\alpha]}(p))$ be the fiber over $F_{[\alpha]}(p)$.
Since the image of $F_{[\alpha]}$ is a normal variety, Zariski's Main Theorem
shows that all irreducible components of $S$ are positive-dimensional,
so there is at least one such component $S' \subset E_j$ which contains $p$.
Then $\alpha$ is a smooth semipositive form in the class $[\alpha]$
and $\alpha|_{S'}\equiv 0$ since $S'$ is contained in a fiber of $F_{[\alpha]}$ and $\alpha$ is the pull back of Fubini-Study metric.
But this means
that $(\alpha |_{E_j} )^{n-1}(p) = 0$, since $\alpha |_{E_j}$ has zero eigenvalues in all directions
tangent to $S��$. Hence, this is true for all $p$ in a Zariski open subset
of $E_j$. We conclude that $[\alpha^{n-1}]\cdot[E_j ] =
\int_{E_j} (\alpha |_{E_j} )^{n-1}= 0$.
\vspace{2mm}

Now we can prove Theorem \ref{PCY}.
\begin{proof}
By Lemma \ref{tos07}, there exists a semipositive $(1,1)$-form $\alpha_0\in[\alpha]$ such that $\alpha_0$ is
strictly positive outside a subvariety $V$. If $V$ is of codimension greater than one, Proposition \ref{dbc prop 1}
implies that $[\alpha^{n-1}]=[\alpha_0^{n-1}]$ is a balanced metric. If $V$ is of codimension one
with irreducible components
$E_1,\cdots,E_k$, then $[\alpha^{n-1}]\cdot [E_j]=0$ for all $1\leq j\leq k$, thus $[\alpha^{n-1}]\notin\mathcal B$. On the other hand, the converse is obvious.

Next, let's prove $[\alpha^{n-1}]\in \mathcal B$ implies that $[\alpha]$ is a big class. Otherwise, we would have
$\int_X\alpha^n=0$. Since $[\alpha]$ is nef,
there exists a positive current $T\in[\alpha]$. Hence
$$\int_X\alpha^{n-1}\wedge T=\int_X\alpha^n=0.$$
Then $[\alpha^{n-1}]\in \mathcal B$ implies $T=0$. Thus $[\alpha]=[T]=0$. This is  a contradiction.
\end{proof}

We are going to give some examples which show that the holomorphic maps $F_{[\alpha]}$ contract
high codimensional subvarieties to points, so we can
apply Theorem \ref{nefb}. The first one is known as a conifold in the physics literature
\cite{GMS95} (see also \cite{R06}). We learned this from \cite{T07}. Let $X_0$ be a nodal quintic
in $\mathbb P^4$ which has 16 nodal points.
Then a smooth Calabi-Yau manifold $X$ is given by a small resolution $f:X\to X_0$, that is
a birational morphism which is an isomorphism outside the preimages of the nodes, which are 16 rational curves.
Thus we get a contracting map
from $X$ to $\mathbb P^4$. It is easy to see that the pullback of the Fubini-Study metric
is our desired form.

There are also other examples from algebraic geometry (cf. \cite{OD01}, page 24-26).
Let $r$ and $s$ be positive integers, let $E$ be the vector bundle on $\mathbb{P}^{s}$
associated to the locally free sheaf $\mathcal{O}_{\mathbb{P}^{s}}\oplus \mathcal{O}_{\mathbb{P}^{s}}(1)^{r+1}$,
and let $Y_{r,s}$ be the smooth $(r+s+1)$-dimensional variety $\mathbb{P}(E^{*})$. The projection
$\pi:Y_{r,s}\to \mathbb{P}^{s}$ has a section $P_{r,s}$ corresponding to the trivial quotient of $E$.
The linear system $|\mathcal{O}_{Y_{r,s}}(1)|$
is base point free. Hence it induces a holomorphic map:
\begin{eqnarray*}
  C_{r,s}: Y_{r,s}\rightarrow \mathbb{P}^{(r+1)(s+1)}.
\end{eqnarray*}
Moreover, $C_{r,s}$ contracts $P_{r,s}$ to a point and is an immersion on its complement.
And its image is the cone over the Segre embedding of $\mathbb{P}^{r}\times \mathbb{P}^{s}$.

Thus, the pull-back of the Fubini-Study metric of $\mathbb{P}^{(r+1)(s+1)}$ is a
smooth $(1,1)$-form $\alpha=C_{r,s}^{*}\omega_{FS}$.  Clearly $\alpha$ is pointwise
nonnegative on the whole $Y_{r,s}$ and is strictly positive outside $P_{r,s}$ with codimension $r+1$.
Thus $[\alpha^{r+s}]$ is a balanced class on $\mathbb{P}(E^{*})$.
Furthermore, $\int_{P_{r,s}}\alpha^{s}=0$ implies $\alpha\in\partial\mathcal{K}(Y_{r,s})$.

In fact, there are a lot of such examples in the Minimal Model Program, encountered when dealing with contraction
maps of flipping type (\cite{KMM87}).
\vspace{2mm}

The following comment has been formulated by V. Tosatti. In order to produce more examples of birational contraction
morphisms as in Lemma \ref{tos07}, one can take $X$ more generally  to be any smooth
projective variety with $-K_X$ nef. This class includes not only Calabi-Yau but also
Fano manifolds. Under this assumption, if $L$ is any line bundle on $X$ which
is nef and big, then Kawamata's base-point-free theorem again gives us
that $L$ is semi-ample and so there is a birational contraction $F_L$
exactly as in Lemma \ref{tos07}. It also works for
$\mathbb R$-linear combinations of line
bundles (i.e. big classes on the boundary of $\mathcal K_{NS}$), because again the
big points on the boundary of $\mathcal K_{NS}$ are locally rational polyhedral
(if $X$ is Fano, then the whole boundary of $\mathcal K_{NS}$ is rational polyhedral). Thus if $X$ has nef anticanonical bundle, we can still apply Theorem \ref{nefb}.

\section{Characterization theorem on a nef class being K\"ahler}
Using a similar method as in Section 2, we can characterize when a nef class
$[\alpha]$ is K\"ahler under the assumption that
$[\alpha^{n-1}]$ is a balanced class.
\begin{theo}
Let $X$ be a compact $n$-dimensional K\"ahler manifold and $\eta$
a smooth volume form of $X$. Assume that $[\alpha]$ is a nef  class
such that $[\alpha^{n-1}]$ is a balanced class (so $[\alpha]$ is  big). If there exists a
balanced metric $\tilde\omega$ in $[\alpha^{n-1}]$ (i.e., $\tilde\omega^{n-1}\in [\alpha^{n-1}]$) such that
$c_{\tilde\omega}\geq c_{\alpha}$ with $c_{\tilde\omega}=\min_{X}\frac{\tilde\omega^n}{\eta}$
and $c_\alpha=\frac{\int_X\alpha^n}{\int_X\eta}$,
then $[\alpha]$ is a K\"ahler class.
\end{theo}
\begin{proof}
Since $\tilde{\omega}^{n-1}\in[\alpha^{n-1}]$, there exists a smooth
$(n-2,n-2)$-form $\phi$ such that
\begin{equation*}
   \tilde{\omega}^{n-1}=\alpha^{n-1}+i\partial\bar{\partial}\phi>0.
\end{equation*}
Fix a K\"ahler metric $\omega$ on $X$. Then for $0<t\ll1$,
$$(\alpha+t\omega)^{n-1}+i\partial\bar{\partial}\phi=
\tilde{\omega}^{n-1}+O(t)>0.$$ Thus there exists a balanced metric
$\tilde\omega_t$ such that
\begin{equation}\label{omegat}
    \tilde{\omega}_{t}^{n-1}=(\alpha+t\omega)^{n-1}+i\partial\bar{\partial}\phi
\end{equation}
and $\tilde\omega_0=\tilde\omega$. Clearly,
as $t\to 0$, $\tilde\omega_t\to \tilde\omega$ in $C^\infty(\Lambda^{1,1}(X))$.
Then if we let $F_{\tilde\omega_t}:=\frac{\tilde\omega_t^n}{\eta}$,
we have
\begin{equation*}\label{lim balance}
    F_{\widetilde{\omega}_{t}}\to F_{\widetilde{\omega}}
\end{equation*}
in $ C^\infty(X)$ as $t\to 0$.

On the other hand, since $[\alpha+t\omega]$ is a K\"ahler class, by Lemma \ref{yau} there exists a family of smooth
functions $u_t$ such that $\alpha+t\omega+i\partial\bar\partial u_t$ is K\"ahler and
\begin{equation*}
(\alpha+t\omega+i\partial\bar\partial u_t)^n=c_t\eta
\end{equation*}
with $c_t=\frac{\int_X(\alpha+t\omega)^n}{\int_X\eta}$.
Moreover, by Lemma \ref{begz10}, there also exists an $\alpha$-psh function $u_0$ such that
\begin{eqnarray*}
\bigl\langle(\alpha+i\partial\bar\partial u_0)^n\bigr\rangle=c_\alpha\eta.
\end{eqnarray*}
Such $u_t$ and $u_0$ satisfy the following relations
\begin{equation*}
\alpha+t\omega+i\partial\bar\partial u_t\to \alpha+i\partial\bar\partial u_0\ \ \ \ \textup{as currents on $X$}
\end{equation*}
and
\begin{equation}\label{ddu}
\alpha+t\omega+i\partial\bar\partial u_t\to \alpha+i\partial\bar\partial u_0\ \
\ \ \textup{in $C^\infty_{\textup{loc}}(\textup{Amp}(\alpha))$}.
\end{equation}
We denote $\alpha_t=\alpha+t\omega+i\partial\bar\partial u_t$ and $\alpha_0=\alpha+i\partial\bar\partial u_0$.
Then from (\ref{omegat}), we have
 \begin{equation}\label{phit}
\tilde\omega_t^{n-1}=\alpha_t^{n-1}+i\partial\bar\partial\phi_t
\end{equation}
for some smooth $(n-2,n-2)$-form $\phi_t$ on $X$.

By the above notations, we have
$$\frac{F_{\widetilde{\omega}_{t}}}{c_{t}}=\frac{\widetilde{\omega}_{t}^{n}}{\alpha_{t}^{n}}.$$
We apply the AM-GM inequality to obtain
\begin{eqnarray*}\label{a g ineq}
\begin{aligned}
   \Bigl(\frac{F_{\tilde{\omega}_{t}}}{c_{t}}\Bigr)^{\frac{n-1}{n}}
   =& \Bigl(\frac{\det(\alpha_{t}^{n-1}+i\partial\bar{\partial}\phi_{t})}{\det\alpha_{t}^{n-1}}\Bigr)^{\frac{1}{n}}\\
   \leq & 1+\frac{1}{n}\sum_{k,l}(\alpha_{t}^{n-1})^{k\bar l}(i\partial\bar{\partial}\phi_{t})_{k\bar l}.
   \end{aligned}
\end{eqnarray*}
Equivalently, we have
\begin{equation}\label{a g ineq form}
   \Bigl(\frac{F_{\widetilde{\omega}_{t}}}{c_{t}}\Bigr)^{\frac{n-1}{n}}\alpha_{t}^{n}
   \leq \alpha_{t}^{n}+\alpha_{t}\wedge i\partial\bar{\partial}\phi_{t}.
\end{equation}

We deal with the second term in the above equality, namely
$$\alpha_{t}\wedge
i\partial\bar{\partial}\phi_{t}=\alpha_{t}\wedge\tilde{\omega}_{t}^{n-1}-\alpha_{t}^{n}.$$
As discussed in the proof of
Theorem \ref{bignef}, there exists a convergent subsequence $\alpha_{t_k}\wedge \tilde\omega_{t_k}^{n-1}$
of measures $\alpha_t
\wedge\tilde\omega_t^{n-1}$ and a convergent sequence $\alpha_{t_k}^{n}$ of measures $\alpha_t^n$.
If we denote their limits by $\mu_{1}$ and
$\mu_{2}$, and denote $\mu_0=\mu_1-\mu_2$, then we have
\begin{equation*}
    \alpha_{t_k}\wedge i\partial\bar{\partial}\phi_{t_k}\to\mu_{0} \ \ \ \ \textup{as currents}.
\end{equation*}
Letting $t=t_k$ in   (\ref{a g ineq form}), integrating with respect to any positive
smooth function, and letting $t_k$ go to
zero, we find that the condition $c_{\tilde\omega}\geq c_0$  implies that $\mu_{0}$ is a
positive measure.

Meanwhile, since \begin{eqnarray*}
\begin{aligned}
    \int_{X}\mu_{0}=&\lim_{t\to 0}\int_{X}\alpha_{t}\wedge\tilde{\omega}_{t}^{n-1}-\alpha_{t}^{n}\\
    =&\int_{X}\alpha\wedge(\tilde{\omega}^{n-1}-\alpha^{n-1}),
    \end{aligned}
\end{eqnarray*}
and as $\alpha$ is nef and $\tilde\omega^{n-1}\in [\alpha^{n-1}]$, we have $\int_X\mu_0=0$.
Thus $\mu_{0}=0$ and  $F_{\tilde\omega}=c_\alpha$ pointwise.

On $\textup{Amp}(\alpha)$, we define a smooth $(1,1)$-form
$$\Psi_0=\lim_{t\to 0}i\partial\bar\partial\phi_t.$$
Then from (\ref{phit}), (\ref{ddu}) and (\ref{omegat}),
we have
\begin{equation*}
\Psi_0=\lim_{t\to 0}(\tilde\omega_t^{n-1}-\alpha_t^{n-1})=\tilde\omega^{n-1}-\alpha_0^{n-1}.
\end{equation*}
Hence by uniqueness of the limit, we have on $\textup{Amp}(\alpha)$
\begin{equation*}\label{trace zero}
   \alpha_{0}\wedge \Psi_0=0.
\end{equation*}
Since $F_{\tilde\omega}=c_\alpha$, this implies that on $\textup{Amp}(\alpha)$,
\begin{equation*}
    1=\Bigl(\frac{\det\tilde{\omega}^{n-1}}{\det\alpha_{0}^{n-1}}\Bigr)^{\frac{1}{n}}\leq 1+
    \frac{1}{n}\sum_{k,l}(\alpha_{0}^{n-1})^{k\bar l}(\Psi_0
)_{k\bar l}=1.
\end{equation*}
Thus $\Psi_0=0$. Therefore
$\tilde{\omega}^{n-1}=\alpha_{0}^{n-1}$ or
$\tilde{\omega}=\alpha_{0}$ on $\textup{Amp}(\alpha)$.

Since $\tilde{\omega}$ is smooth on $X$ and $d\tilde\omega=d\alpha_0=0$
on $\textup{Amp}(\alpha)$, by continuity, $d\tilde\omega=0$ on $X$, i.e.,
$\tilde\omega$ is a K\"ahler metric on $X$.
However, since $[\tilde\omega^{n-1}]=[\alpha^{n-1}]$, by Theorem \ref{inj_thm}, $[\tilde\omega]=[\alpha]$.
Thus $[\alpha]$ is a K\"ahler calss.
\end{proof}

Now we are in a position to conclude the proof of Theorem \ref{chara}.
\begin{proof}
We assume that  there exists a solution $\Omega_{CY}\in[\alpha^{n-1}]$ to equation (\ref{eon})
 for $c\leq (\int_X\alpha^n)^{-1}$.
We write $\Omega_{CY}=\tilde\omega^{n-1}$ and then compute
$$\frac{\tilde\omega^n}{\omega_0^n}=\frac{\parallel\zeta\parallel_{\omega_0}^2}{\parallel\zeta\parallel^2_{\tilde\omega}}
=\frac 1 {\parallel\zeta\parallel_{\Omega_{CY}}}=\frac 1 c \geq \frac{\int_X\alpha^n}{\int_X\omega_0^n}.
$$
Hence we can use the above theorem. Thus  $[\alpha]$ is a K\"ahler class. Now the proof follows from
Theorem \ref{nefb} in
\cite{FWW10}.
\end{proof}

\section{Appendix}
In this appendix, we show that the conjectured cone duality $\mathcal{E}^ \vee =\overline{\mathcal{M}}$ in \cite{BDPP13} implies that the movable cone $\mathcal{M}$ coincides with the balanced cone $\mathcal{B}$.
Let us first recall the definitions of the pseudoeffective cone and the movable cone of  a K\"ahler manifold.
\begin{defi}
\label{cone_def}
Let $X$ be an $n$-dimensional compact K\"ahler manifold.\\
(1) The \textup{pseudoeffective cone} $\mathcal{E}\subset H^{1,1}_{BC}(X, \mathbb{R})$ is defined to be the convex cone generated by all positive $d$-closed $(1,1)$-currents.\\
(2) The \textup{movable cone} $\mathcal{M}\subset H^{n-1,n-1}_{BC}(X, \mathbb{R})$ is defined to be the convex cone generated by all positive $d$-closed $(n-1,n-1)$-currents of the form $\mu_*(\widetilde{\omega}_1 \wedge...\wedge \widetilde{\omega}_{n-1})$, where $\mu$ ranges among all K\"ahler modifications from some $\widetilde{X}$ to $X$ and $\widetilde{\omega}_i$'s are K\"ahler metrics on $\widetilde{X}$.
\end{defi}

In \cite{MT09}, Toma observed that every movable curve on a projective manifold can be represented by a balanced metric under the assumption $\mathcal{E}^ \vee =\overline{\mathcal{M}}$. We observe that Toma's result holds for all movable classes on a compact K\"ahler manifold. Its proof is along the lines of \cite{MT09} and the  arguments go through {\sl mutatis mutandis}.

\begin{theo}
\label{bcone_mcone_thm}
Let $X$ be an $n$-dimensional compact K\"ahler manifold. Then
$\mathcal{E}^ \vee =\overline{\mathcal{M}}$ implies $\mathcal{M}=\mathcal{B}$
\end{theo}
\begin{proof}
In Remark 3.4, we have proved the cone duality $\mathcal{E}_{dd^c}^ \vee =\overline{\mathcal{B}}$. Hence, we first  prove $\mathcal{E}_{dd^c}^ \vee =
\mathcal{E}^ \vee$.
By the $\partial\bar\partial$-lemma, the natural homomorphism $j: H^{1,1}_{BC}(X, \mathbb{R})\rightarrow V^{1,1}(X, \mathbb{R})$ is actually an isomorphism (see \cite{AB93}). Hence when $j$ is restricted on $\mathcal{E}$ (which is also denoted  by $j$),
 $j: \mathcal{E}\rightarrow \mathcal{E}_{dd^c}$ is injective. We should show that $j$ is also surjective. For any $[T]_{dd^c}\in \mathcal{E}_{dd^c}$ with $T$ positive, there exists some current $S$ such that $d(T+\partial \bar S +\bar \partial S)=0$. We claim that the class $[T+\partial \bar S +\bar \partial S]$ is pseudoeffective, i.e., $[T+\partial \bar S +\bar \partial S]\in  \mathcal{E}$.  We need a result in \cite{AB95}, which states that for any modification $\mu: \widetilde{X}\rightarrow X$ and any positive $dd^c$-closed $(1,1)$-current $T$ on $X$, there exists an unique positive $dd^c$-closed $(1,1)$-current $\widetilde{T}$ on $\widetilde{X}$ such that $\mu_* \widetilde{T}=T$ and $\widetilde{T}\in \mu^* [T]_{dd^c}$. Now, take a smooth $(1,1)$-form $\alpha\in [T+\partial \bar S +\bar \partial S]$ (which will also be a representative of $[T]_{dd^c}$), $\widetilde{T}\in \mu^* [T]_{dd^c}$ implies that there exists some current $\widetilde{S}$ such that $\widetilde{T}= \mu^* \alpha +\partial \overline{\widetilde{S}} +\bar \partial{\widetilde{S}}$. Thus, for any modification $\mu: \widetilde{X}\rightarrow X$ with $\widetilde{X}$ being K\"ahler, we have
\begin{align*}
 \label{int_eq}
 \int_{X} \alpha \wedge \mu_{*}(\widetilde{\omega}_1 \wedge...\wedge \widetilde{\omega}_{n-1})
 &=\int_{\widetilde{X}}\mu^* \alpha\wedge \widetilde{\omega}_1 \wedge...\wedge \widetilde{\omega}_{n-1}\\
 &=\int_{\widetilde{X}}(\mu^* \alpha+\partial \bar{\widetilde{S}} +\bar \partial{\widetilde{S}})\wedge \widetilde{\omega}_1 \wedge...\wedge \widetilde{\omega}_{n-1}\\
 &=\int_{\widetilde{X}}\widetilde{T}\wedge \widetilde{\omega}_1 \wedge...\wedge \widetilde{\omega}_{n-1}\\
 &\geq 0.
\end{align*}
By the
arbitrariness of $\mu$ and $\widetilde{\omega}_i$'s, $\mathcal{E}^ \vee =\overline{\mathcal{M}}$ indicates that $[T+\partial \bar S +\bar \partial S]\in \mathcal{E}$. This confirms the surjectivity of $j: \mathcal{E}\rightarrow \mathcal{E}_{dd^c}$, and hence $j$ is an isomorphism.

Now, it is easy to see that $\mathcal{M}=\mathcal{B}$. On one hand, since any balanced metric takes positive values on $\mathcal{E}\backslash \{0\}$, $\mathcal{B}$ is obviously contained in the interior of $\mathcal{E}^ \vee$, thus $\mathcal{B}\subseteq \mathcal{M}$. On the other hand, $j(\mathcal{E})=\mathcal{E}_{dd^c}$ yields any movable class taking positive values on $\mathcal{E}_{dd^c} \backslash \{0\}$, hence $\mathcal{E}_{dd^c} ^\vee = \overline{\mathcal{B}}$ implies $\mathcal{M} \subseteq \mathcal{B}$. Thus, we obtain $\mathcal{B}=\mathcal{M}$.
\end{proof}

%\begin{rema}
%\label{bc_sg_rmk}
%From the above argument, it is easy to see we indeed have $\mu_* \widetilde{\mathcal{B}}\subseteq \mathcal{B}$ under the assumption $\mathcal{E}^ \vee = %\overline{\mathcal{M}}$. In this direction, the second named author \cite{Xia13} has shown that $\mu_* \widetilde{\mathcal{SG}}=\mathcal{SG}$ where %$\mathcal{SG}$ is the strongly Gauduchon cone generated by strongly Gauduchon metrics. Here, a strongly Gauduchon metric means a $d$-closed smooth $2n-2$-form %with strictly positive $(n-1,n-1)$-part. Thus, unconditionally, any push-forward of a balanced metric by a proper modification can be modified by a $d$-exact %current to be a strongly Gauduchon metric.
%\end{rema}
\begin{rema}
\label{bc_sg_rmk}
 In \cite{BDPP13}, the authors have observed that their conjectured cone duality is true for hyper-K\"ahler manifolds or K\"ahler manifolds which are the limits of projective manifolds with maximal Picard number under holomorphic deformations. So in such cases,  $\mathcal B=\mathcal M$ holds.
\end{rema}
Inspired by the above theorem, we naturally propose the following problem concerning the balanced cone of a general compact balanced manifold.
\begin{conj}
\label{cone_conj}
Let $X$ be a compact balanced manifold. Then $\mathcal{E}^ \vee =\overline{\mathcal{B}}$ holds.
\end{conj}

\end{document}